\documentclass[singlecolumn,times,12pt]{elsarticle}

\usepackage{graphicx}
\usepackage{amssymb}
\usepackage{amsthm}
\usepackage{amsmath}
\usepackage{longtable}
\usepackage{subfig}
\usepackage{pst-all}
\usepackage{relsize}
\usepackage{ifpdf}
\ifpdf%
 \usepackage{pdftricks}
 \begin{psinputs}
    \usepackage{pstricks}
    \usepackage{pstricks-add}
    \usepackage{pst-node}
    \usepackage{pst-grad}
    \usepackage{pst-text}
  \end{psinputs}
\else
  \usepackage{pstricks}
  \usepackage{pstricks-add}
  \usepackage{pst-node}
  \usepackage{pst-grad}
  \usepackage{pst-text}
  \newenvironment{pdfpic}{}{}
\fi
\newtheorem{theorem}{Theorem}
\newtheorem{proposition}[theorem]{Proposition}
\newtheorem{lemma}[theorem]{Lemma}

\newdefinition{definition}[theorem]{Definition}
\newdefinition{assumption}[theorem]{Assumption}
\newdefinition{remark}[theorem]{Remark}
\newdefinition{algorithm}[theorem]{Algorithm}
\newdefinition{example}[theorem]{Example}
\def\R{\mathbb{R}}
\def\N{\mathbb{N}}
\def\bZ{\mathbb{Z}}

\def\X{X}
\def\bX{\mathbb{X}}
\def\Xp{X_p}
\def\bXp{\mathbb{X}_p}
\def\U{U}
\def\bU{\mathbb{U}}
\def\Up{U_p}
\def\bUp{\mathbb{U}_p}

\def\UpNp{U_p^{N_p}}

\def\bUpad{\mathbb{U}_p^{\text{ad}}}

\def\bUpNpad{\mathbb{U}_p^{N_p,\text{ad}}}
\def\bIp{\mathbb{I}_p}
\def\cI{\mathcal{I}}
\def\cIp{\mathcal{I}_p}

\def\cB{\mathcal{B}}
\def\cK{\mathcal{K}}
\def\cL{\mathcal{L}}

\def\cKL{\mathcal{K\hspace*{-1mm}L}}
\def\cP{\mathcal{P}}

\def\Np{N_p}
\def\JN{J^N}

\def\JpNp{J_p^{N_p}}
\def\Jinfty{J^\infty}

\def\Jpinfty{J_p^\infty}
\def\VN{V^N}
\def\VpN{V_p^N}

\def\VpNp{V_p^{N_p}}
\def\Vinfty{V^\infty}

\def\muN{\mu^N}
\def\mupN{\mu_p^N}
\def\mupNp{\mu_p^{N_p}}
\def\fp{f_p}
\def\l{\ell}
\def\lp{\l_p}

\def\i{I}
\def\ip{I_p}
\def\xp{x_p}
\def\xzero{x^0}
\def\xpzero{x_p^0}
\def\xpu{x_p^u}
\def\xref{x^{\text{\rm ref}}}
\def\xpref{x_p^{\text{\rm ref}}}

\def\up{u_p}

\def\upref{u_p^{\text{\rm ref}}}
\def\us{u^*}
\def\ups{u_p^*}

\def\argmin{\mathop{\rm argmin}}

\usepackage{framed}
{\endMakeFramed}
\definecolor{shadecolor}{rgb}{1.,1.,1.}%
\definecolor{framecolor}{rgb}{.0,.0,.0}%

\journal{System \& Control Letters}

\begin{document}
\begin{frontmatter}

\title{Parallelizing a State Exchange Strategy for Noncooperative Distributed NMPC\tnoteref{titlelablel}}
\tnotetext[titlelablel]{This work was supported by the Leopoldina Fellowship Programme LPDS 2009-36.}

\author{J\"{u}rgen Pannek}
\ead{juergen.pannek@googlemail.com}
\address{Faculty of Aerospace Engineering \\ University of the Federal Armed Forces \\ 85577 Munich, Germany}

\begin{abstract}
	We consider a distributed non cooperative control setting in which systems are interconnected via state constraints. Each of these systems is governed by an agent which is responsible for exchanging information with its neighbours and computing a feedback law using a nonlinear model predictive controller to avoid violations of constraints. For this setting we present an algorithm which generates a parallelizable hierarchy among the systems. Moreover, we show both feasibility and stability of the closed loop using only abstract properties of this algorithm. To this end, we utilize a trajectory based stability result which we extend to the distributed setting.
\end{abstract}


\end{frontmatter}

\section{Introduction}

Distributed control problems can arise either naturally, i.e. by a set of coupled systems which shall be controlled, see, e.g., \citet{DS2009}, or if a large problem is decomposed into smaller, again coupled problems, see \citet[Chapter 10]{RM2009} or \citet{S2009} for an overview. In the latter case, the general idea is that smaller problems are solvable easier and faster which allows to even overcompensate the computational effort to coordinate these systems, cf. \citet[Section 7]{RH2007}. In either case, one distinguishes between cooperative control which features a centralized objective, and non cooperative control where the objectives of the systems are independent from each other. Using a centralized objective there are several possibilities to divide the optimization problem into subproblems. If suitable conditions hold then similar performance of the distributed control obtained from these subproblems and of the centralized control can be shown, see, e.g., \citet[Chapter 10]{RM2009} or \citet{GR2010}.

Throughout this work we focus on the non cooperative control setting of systems driven by independent dynamics and control objectives, but coupled by constraints. For each system we impose an agent which exchanges state information with its neighbours and uses its local objective to compute a local control which satisfies the coupling constraints. For the computing task we focus on feedback design via nonlinear model predictive controller (NMPC) which minimizes the distance of the current state to the desired equilibrium over a finite time horizon. To show asymptotic stability of an NMPC closed loop, one often imposes additional stabilizing terminal constraints and costs, see, e.g., \citet{KG1988} or \citet{ChAl1998} respectively. Since such terminal constraints may require long optimization horizons, we focus on the plain NMPC setting without those modifications. In the non distributed case, stability for such problems has been shown in \citet{GPSW2010} whereas the distributed case is treated in \citet{GW2010} using the algorithm of Richards and How \cite{RH2004, RH2007}.

Here, we first prove a stability idea outlined in \citet{GW2010} using the trajectory based setting of \citet[Chapter 7]{GP2011}. This proof allows us to reduce the horizon length in the distributed case while maintaining suboptimality estimates and stability like behavior of the closed loop. Secondly, since the computing time of the NMPC control law for each agent is not negligible, we present an algorithm which allows us to execute these computations in parallel using priority and deordering rules as well as a decision memory. From \citet[Chapter 10]{RM2009} it is known that for the non cooperative control setting one can only expect to reach a Nash equilibrium. Although such a solution may be far from the optimal centralized solution, the closed loop solutions may still be stable and maintain the coupling constraints. For the proposed algorithm we present conditions under which feasibility of the closed loop is guaranteed and present necessary as well as sufficient conditions for asymptotic stability using only abstract properties of both the priority and the deordering rule. While here we focus on the plain NMPC case, we also outline how feasibility and stability results can be obtained using NMPC with terminal constraints or cost.

The paper is organized as follows: First, in Section \ref{Section:Setup and Preliminaries} we formally define the problem under consideration for which we show different stability results for the distributed case in Section \ref{Section:Stability}. In the central Section \ref{Section:The covering algorithm}, we present a covering algorithm which allows us to generate a hierarchy of agents and run the computations of each hierarchy level in parallel. Using this algorithm, we show necessary and sufficient conditions for feasibility and stability of the resulting closed loop and also how much parallelism can be achieved. Instead of a separated example section, we use an analytical example throughout the entire work to present the improvement of the stability result but also to illustrate the abstract functions used within the proposed algorithm in Section \ref{Section:The covering algorithm}. Finally, we draw conclusions in Section \ref{Section:Conclusion} and present ideas for future research based on the presented work.

\section{Setup and Preliminaries}
\label{Section:Setup and Preliminaries}

Throughout this work we consider a set of nonlinear discrete time systems
\begin{align}
	\label{Setup and Preliminaries:system p}
	\xp(n + 1) = \fp(\xp(n), \up(n)), \qquad p \in \cP := \{ 1, \ldots ,P\}, n \in \N_0
\end{align}
with $\xp (n) \in \Xp$ and $\up(n) \in \Up$ and $\N_0$ denoting the set of natural numbers including zero. Here, $\Xp$ and $\Up$, $p \in \cP$, are assumed to be arbitrary metric spaces denoting the state space and the set of admissible control values of the $p$-th system, respectively. The metrics to measure distances between two elements of $\Xp$ or of $\Up$ are denoted by $d_{\Xp}: \Xp \times \Xp \to \R_{\geq 0}$ and $d_{\Up}: \Up \times \Up \to \R_{\geq 0}$ where $\R_{\geq 0}$ denotes the positive reals including zero. In the following we denote the solution of a system $p$ of \eqref{Setup and Preliminaries:system p} corresponding to the initial value $\xp(0) = \xpzero$ and the control sequence $\up(k) \in \Up$, $k = 0, 1, 2, \ldots$, by $\xpu(k, \xpzero)$. 

In order to define our goal we say that a continuous function $\alpha: \R_{\geq 0} \rightarrow \R_{\geq 0}$ is of class $\cK_\infty$ if it satisfies $\alpha(0) = 0$, is strictly increasing and unbounded. A continuous function $\gamma: \R_{\geq 0}^P \to \R_{\geq 0}$ is called a class $\cK_\infty^P$ function if it satisfies $\gamma(0) = 0$, is strictly increasing in each component and is unbounded. A continuous function $\beta: \R_{\geq 0} \times \R_{\geq 0} \rightarrow \R_{\geq 0}$ is of class $\cKL$ if it is strictly decreasing in its second argument with $\lim_{t \rightarrow \infty} \beta(r, t) = 0$ for each $r > 0$ and satisfies $\beta(\cdot, t) \in \cK_\infty$ for each $t \geq 0$. Moreover, $\cB_r(x)$ denotes the open ball with center $x$ and radius $r$ and for arbitrary $x_1, x_2 \in \X$ we denote the distance from $x_1$ to $x_2$ by $\| x_1 \|_{x_2} = d_{\X}(x_1, x_2)$.

For the set of systems \eqref{Setup and Preliminaries:system p} the overall system is given by
\begin{align}
	\label{Setup and Preliminaries:system}
	x(n + 1) = f(x(n), u(n)), \qquad n \in \N_0
\end{align}
with state $x(n) = ( x_1(n)^\top, \ldots, x_P(n)^\top)^\top \in \X = \X_1 \times \ldots \times \X_P$ and control $u(n) = (u_1(n)^\top, \ldots, u_P(n)^\top)^\top \in \U =\U_1 \times \ldots \times \U_P$. Now, our goal is to asymptotically stabilize system \eqref{Setup and Preliminaries:system} at a desired equilibrium point $\xref \in \X$, i.e. to fulfill the following:

\begin{definition}\label{Setup and Preliminaries:def:stability}
	Let $\xref \in \X$ be an equilibrium for a system \eqref{Setup and Preliminaries:system}, i.e., there exists $u \in \U$ such that $f(\xref,u)=\xref$. Then we say that $\xref$ is {\em locally asymptotically stable} if for a given control sequence $(u(n))_{n \in \N_0}$ there exist $r>0$ and a function $\beta \in \cKL$ such that the inequality
	\begin{equation}
		\label{Setup and Preliminaries:def:stability:eq1}
		\| x^u(n, x^0) \|_{\xref} \leq \beta( \| \xzero \|_{\xref}, n)
  	\end{equation}
	holds for all $\xzero \in \cB_r(\xref)$ and all $n\in\N_0$. 
\end{definition}

Additionally, the solution $x^u(\cdot, \xzero)$ shall satisfy state and control constraints. Throughout this work, we incorporate such constraints by considering suitable subsets of the overall state and control value space $\bX \subset \X$, $\bU \subset \U$ for system \eqref{Setup and Preliminaries:system}. As a result, systems \eqref{Setup and Preliminaries:system p} are coupled via the constraint sets $\bX$ and $\bU$ although the respective dynamics are decoupled. The following example illustrates this setting and will be used throughout this paper.

\begin{example}\label{Setup and Preliminaries:example}
	Consider two cars attempting to cross a one lane brigde, i.e. one car has to wait, cf. Figure \ref{Setup and Preliminaries:fig:carexample} for an illustration. Suppose the discrete time dynamics of the cars are given by
	\begin{align*}
		\xp(n + 1) = \xp(n) + \up(n)
	\end{align*}
	with $\up \in \bUp = \{-1, 0, 1\}^2$ and $\xp \in \bZ^2$ for $p = 1, 2$. Since the cars shall not collide, we obtain the restriction
	\begin{align}
		\label{Setup and Preliminaries:example:collision constraint}
		( x_{1,1}, x_{1,2} )^\top \not = (x_{2,1}, x_{2,2} )^\top.
	\end{align}
	The one lane bridge additionally imposes the constraints
	\begin{align}
		\label{Setup and Preliminaries:example:blocking constraint}
		& x_{p,2} = 0 \text{ if } x_{p,1} = 0 \; \text{for $p = 1, 2$} \quad \text{and} \\
		\label{Setup and Preliminaries:example:passing constraint}
		& \left( \begin{array}{c} x_{1,1} + u_{1,1} \\ x_{1,2} + u_{1,2} \end{array} \right) \not = \left( \begin{array}{c} x_{2,1} \\ x_{2,2} \end{array} \right), \quad
		\left( \begin{array}{c} x_{1,1} \\ x_{1,2} \end{array} \right) \not = \left( \begin{array}{c} x_{2,1} + u_{2,1} \\ x_{2,2} + u_{2,2} \end{array} \right)
	\end{align}
	which together form the set $\bX$. Hence, the local set of admissible moves may depend on which car is allowed to drive first, cf. Figures \ref{Setup and Preliminaries:fig:carexample0} and \ref{Setup and Preliminaries:fig:carexample1}. Note that \eqref{Setup and Preliminaries:example:blocking constraint} is the only local constraint whereas --- if they are considered --  \eqref{Setup and Preliminaries:example:collision constraint} induces an algebraic and \eqref{Setup and Preliminaries:example:passing constraint} a neighbouring dynamic dependent coupling.
	\begin{figure}[!ht]
		\begin{center}
			\subfloat[Admissible moves of $x_2$ if $x_1$ is off the bridge\label{Setup and Preliminaries:fig:carexample0}]{\includegraphics[width=0.32\textwidth]{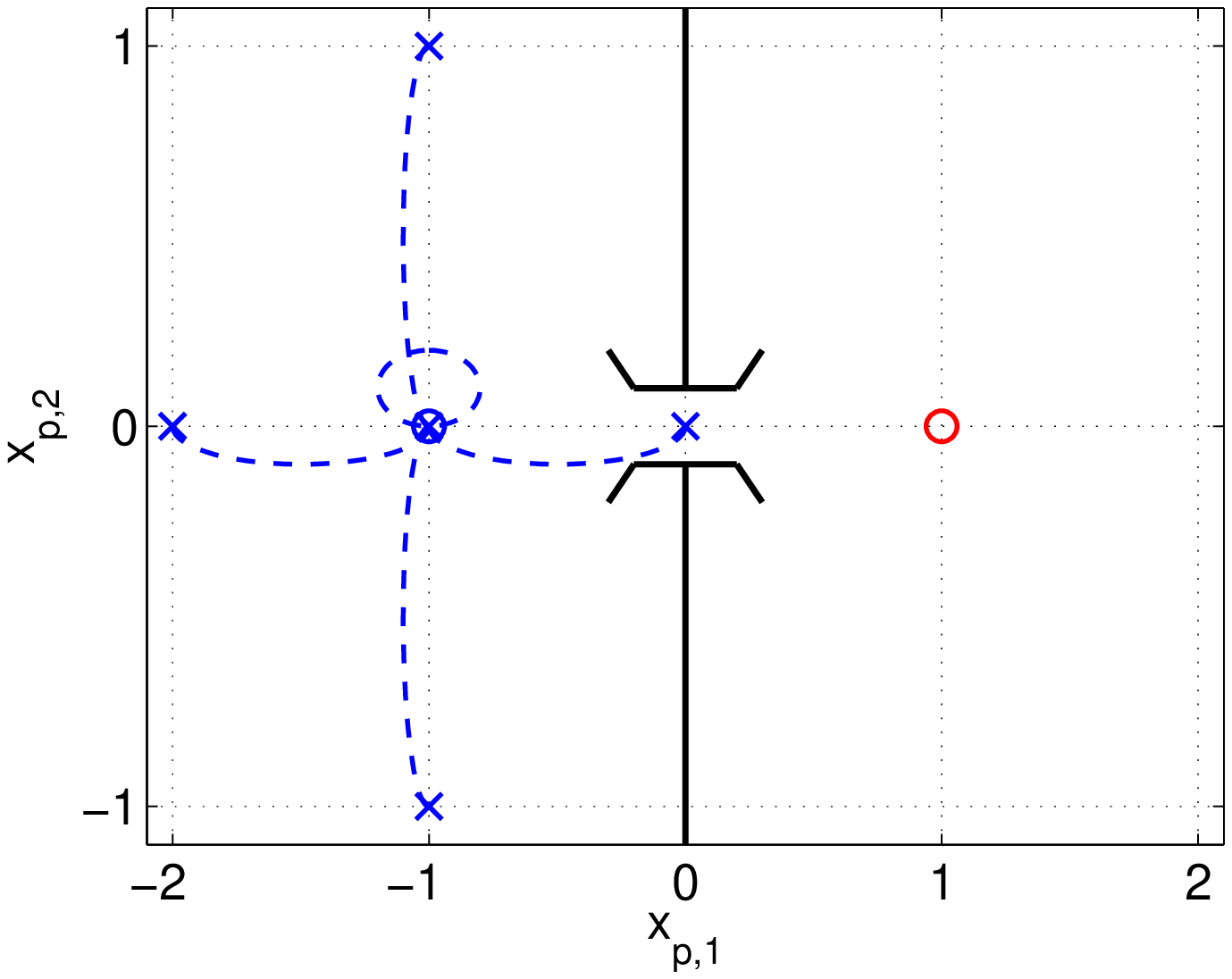}}
			\subfloat[Admissible moves of $x_2$ if $x_1$ enters bridge\label{Setup and Preliminaries:fig:carexample1}]{\includegraphics[width=0.32\textwidth]{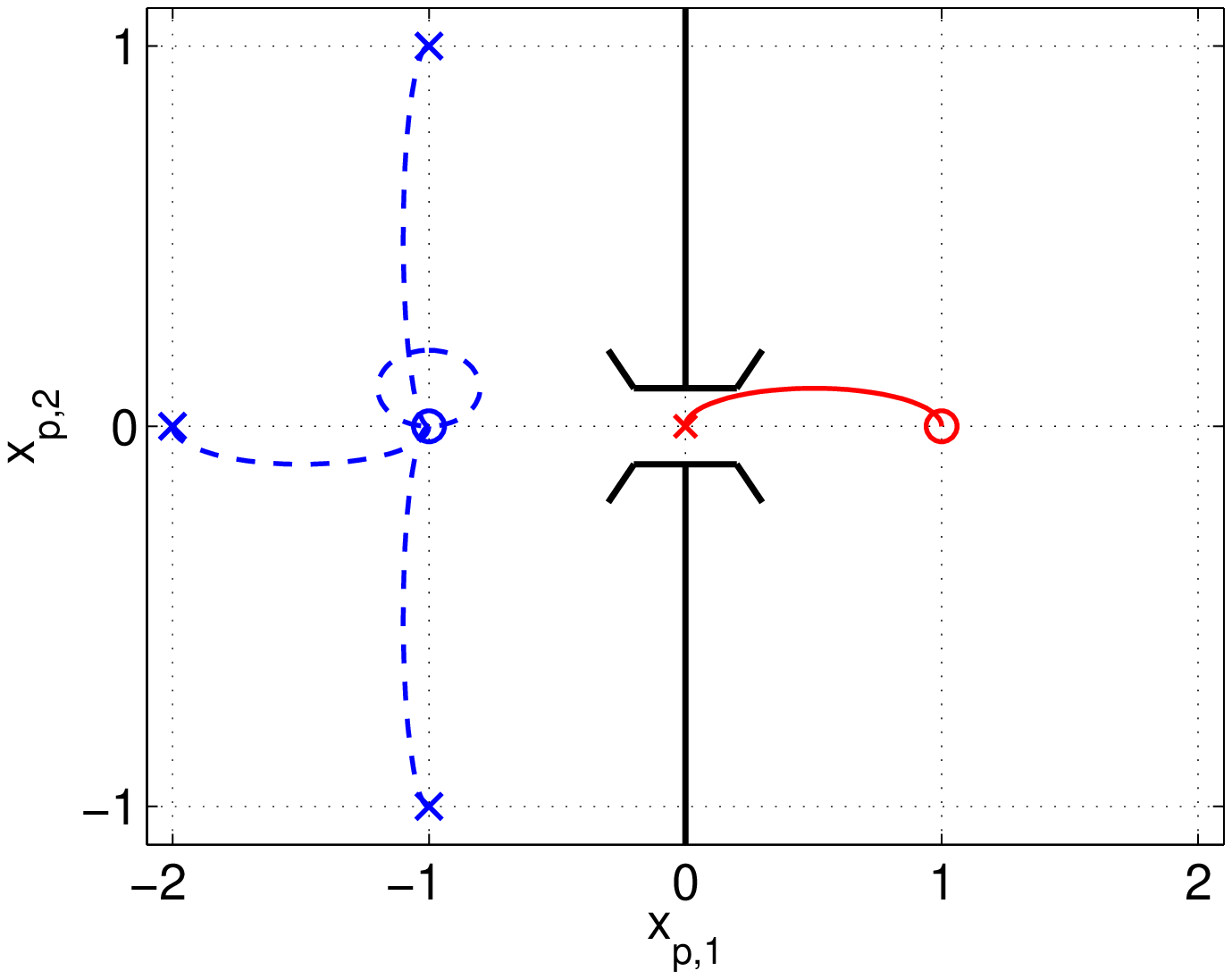}}
			\subfloat[Admissible solution\label{Setup and Preliminaries:fig:carexample2}]{\includegraphics[width=0.32\textwidth]{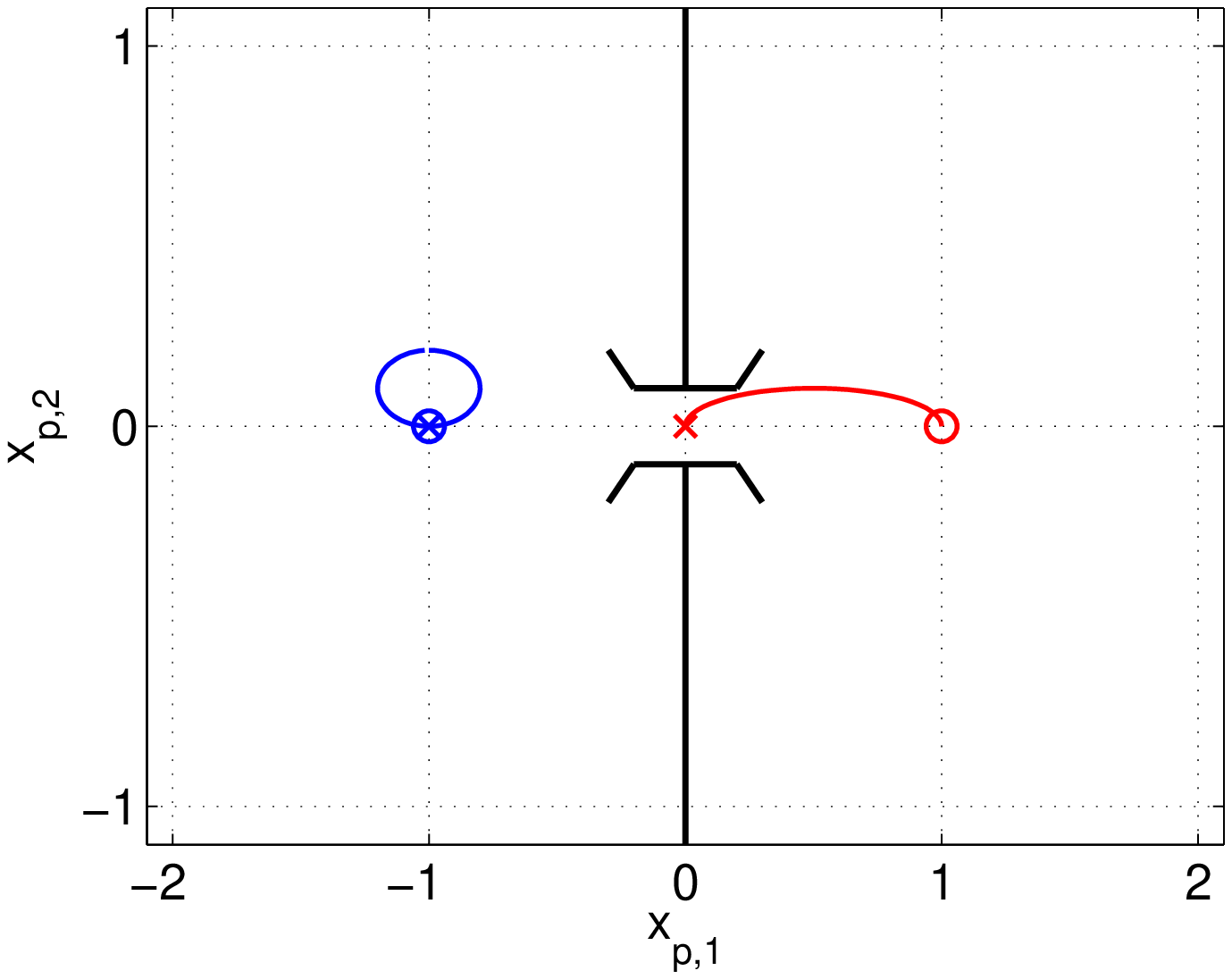}}
		\end{center}
		\caption{Admissible moves for initial conditions $\xzero_{1} = (1, 0)^\top$, $\xzero_{2}=(-1,0)^\top$}\label{Setup and Preliminaries:fig:carexample}
	\end{figure}
\end{example}

The purpose of this work is to show conditions under which stability of the overall systems can be guaranteed by using only local controllers. To this end, we impose an agent for each system $p \in \cP$ to compute a suitable control sequence $\up(\cdot) \in \Up^{\N_0} := \{ \up(k) \mid \up(k) \in \Up \forall k \in \N_0 \}$ and to exchange information with other agents in order to verify that the constraints $\bX \subset \X$, $\bU \subset \U$ are satisfied. Throughout this work each agent computes its control sequence via a nonlinear model predictive controller, a methodology which will be explained after Definition \ref{Setup and Preliminaries:set of admissible control sequences}, below. In order to achieve asymptotic stability of the overall system \eqref{Setup and Preliminaries:system} we develop a covering algorithm to coordinate all agents. The idea of this algorithm is the following: Since some subsystems impose constraints on each other, a priority rule is used to generate a hierarchy among the subsystems. As a result, agents which are on the same hierarchy level can compute their local optimal control in parallel while the hierarchy levels remain in serial. Additionally, a deordering rule is introduced to repeatedly verify if the hierarchy can be flattened, i.e. if more agents can work in parallel. For details on these rules we refer to Section \ref{Section:The covering algorithm}.

Since we want to compute local controls $\up$ 
we must define the local constraints for each single system $\fp$, $p \in \cP$. To this end, we ``project'' the constraint set $\bX$ to the state space of a subset of systems.

\begin{definition}\label{Setup and Preliminaries:def:partial state contraints}
	For an index set $\cIp = \{ p_1, \ldots, p_m \} \subset \cP$ with $m \in \N$, $m \leq P$ and $p_i \not = p_j$ for all $i, j \in \{1, \ldots, m\}$ the \textit{set of partial states} is defined as $\X_{\cIp} = \X_{p_1} \times \ldots \times \X_{p_m}$ and we denote elements of $\X_{\cIp}$ by $x_{\cIp} = (x_{p_1}, \ldots, x_{p_m})$. Accordingly, the \textit{partial state constraint set} is defined by
	\begin{align*}
		\bX_{\cIp} := \{ x_{\cIp} \in \X_{\cIp} \mid \text{there is $\tilde{x} \in \bX$ with $\tilde{x}_{p_i} = x$ for $i = 1, \ldots, m$} \}.
	\end{align*}
\end{definition}

In case of Example \ref{Setup and Preliminaries:example}, Definition \ref{Setup and Preliminaries:def:partial state contraints} basically means that only those constraints induced by the neighbours contained in $\cIp$ have to be considered, i.e. if agent $p = 1$ ignores agent $p=2$, then only constraint \eqref{Setup and Preliminaries:example:blocking constraint} has to be fulfilled. In the general case, this task can be accomplished by using information on neighbouring systems which are available to an agent and allow to generate a hierarchy among the agents. Here, we assume that this information can be exchanged repeatedly in between two time instants $n$ and $n + 1$. 

Throughout this work we consider changing network topologies, i.e. the sets of neighbours at time instants $n$ and $n + 1$ may differ, see Figure \ref{Setup and Preliminaries:fig:communication} for an illustration. 

\begin{figure}[!ht]
	\begin{minipage}{6.0cm}
		\centering
		\begin{pdfpic}
			\psset{xunit=1.0cm,yunit=1.0cm,runit=1.0cm}
			\sffamily
			\begin{pspicture}(0.0,0.2)(4.5,4.2)
				\put(1.0,1.0){\circle*{0.2}}
				\put(4.0,1.0){\circle*{0.2}}
				\put(1.0,4.0){\circle*{0.2}}
				\put(4.0,4.0){\circle*{0.2}}
				\uput[180](1.0,1.0){$x_1$}
				\uput[0](4.0,1.0){$x_2$}
				\uput[0](4.0,4.0){$x_3$}
				\uput[180](1.0,4.0){$x_4$}
				\psline[linestyle=solid,arrowscale=2]{<->}(1.1,1.0)(3.9,1.0)
				\psline[linestyle=solid,arrowscale=2]{<->}(1.1,1.1)(3.9,3.9)
				\psline[linestyle=solid,arrowscale=2]{<->}(1.0,1.1)(1.0,3.9)
				\psline[linestyle=solid,arrowscale=2]{<->}(3.9,1.1)(1.1,3.9)
				\psline[linestyle=solid,arrowscale=2]{<->}(4.0,1.1)(4.0,3.9)
				\psline[linestyle=solid,arrowscale=2]{<->}(3.9,4.0)(1.1,4.0)
				\uput[270](2.5,0.7){Time instant $n$}
			\end{pspicture}
		\end{pdfpic}
	\end{minipage}
	\hfill
	\begin{minipage}{6.0cm}
		\centering
		\begin{pdfpic}
			\psset{xunit=1.0cm,yunit=1.0cm,runit=1.0cm}
			\begin{pspicture}(0.0,0.2)(4.5,4.2)
				\sffamily
				\put(1.0,1.0){\circle*{0.2}}
				\put(4.0,1.0){\circle*{0.2}}
				\put(1.0,4.0){\circle*{0.2}}
				\put(4.0,4.0){\circle*{0.2}}
				\uput[180](1.0,1.0){$x_1$}
				\uput[0](4.0,1.0){$x_2$}
				\uput[0](4.0,4.0){$x_3$}
				\uput[180](1.0,4.0){$x_4$}
				\psline[linestyle=solid,arrowscale=2]{<->}(1.1,1.0)(3.9,1.0)
				\psline[linestyle=solid,arrowscale=2]{<->}(1.0,1.1)(1.0,3.9)
				\psline[linestyle=solid,arrowscale=2]{<->}(4.0,1.1)(4.0,3.9)
				\psline[linestyle=solid,arrowscale=2]{<->}(3.9,4.0)(1.1,4.0)
				\uput[270](2.5,0.7){Time instant $n + 1$}
			\end{pspicture}
		\end{pdfpic}
	\end{minipage}
	\caption{Possible sequence of communication graphs at time instants $n$ and $n + 1$}
	\label{Setup and Preliminaries:fig:communication}
\end{figure}
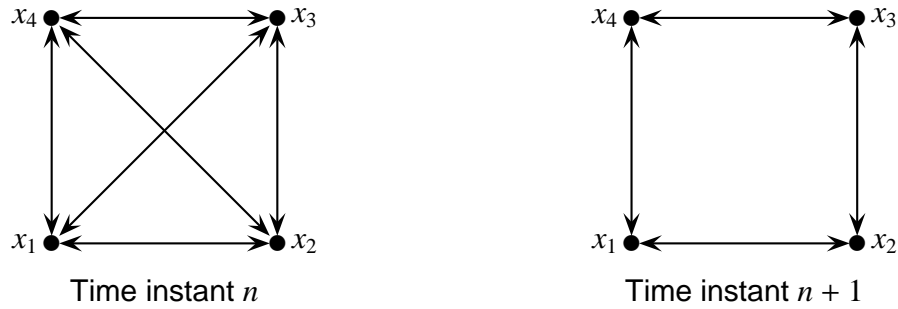

Additionally, we allow the case that even if neighbouring information of a system $q \in \cP \setminus \{ p \}$ is known to an agent $p \in \cP$, agent $p$ ignores that information if his hierarchy level is higher than the level of the neighbour. Consequently, the dependency graph which results from this hierarchy may differ from the communication graph as illustrated in Figure \ref{Setup and Preliminaries:fig:communication and dependency}.

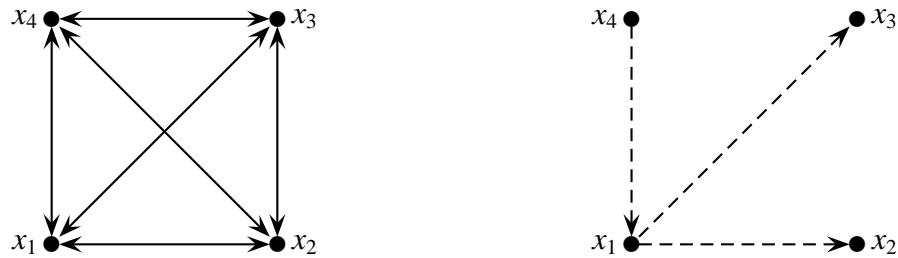
\begin{figure}[!ht]
	\begin{minipage}{6.0cm}
		\centering
		\begin{pdfpic}
			\psset{xunit=1.0cm,yunit=1.0cm,runit=1.0cm}
			\sffamily
			\begin{pspicture}(0.0,0.8)(4.5,4.2)
				\put(1.0,1.0){\circle*{0.2}}
				\put(4.0,1.0){\circle*{0.2}}
				\put(1.0,4.0){\circle*{0.2}}
				\put(4.0,4.0){\circle*{0.2}}
				\uput[180](1.0,1.0){$x_1$}
				\uput[0](4.0,1.0){$x_2$}
				\uput[0](4.0,4.0){$x_3$}
				\uput[180](1.0,4.0){$x_4$}
				\psline[linestyle=solid,arrowscale=2]{<->}(1.1,1.0)(3.9,1.0)
				\psline[linestyle=solid,arrowscale=2]{<->}(1.1,1.1)(3.9,3.9)
				\psline[linestyle=solid,arrowscale=2]{<->}(1.0,1.1)(1.0,3.9)
				\psline[linestyle=solid,arrowscale=2]{<->}(3.9,1.1)(1.1,3.9)
				\psline[linestyle=solid,arrowscale=2]{<->}(4.0,1.1)(4.0,3.9)
				\psline[linestyle=solid,arrowscale=2]{<->}(3.9,4.0)(1.1,4.0)
			\end{pspicture}
		\end{pdfpic}
	\end{minipage}
	\hfill
	\begin{minipage}{6.0cm}
		\centering
		\begin{pdfpic}
			\psset{xunit=1.0cm,yunit=1.0cm,runit=1.0cm}
			\begin{pspicture}(0.0,0.8)(4.5,4.2)
				\sffamily
				\put(1.0,1.0){\circle*{0.2}}
				\put(4.0,1.0){\circle*{0.2}}
				\put(1.0,4.0){\circle*{0.2}}
				\put(4.0,4.0){\circle*{0.2}}
				\uput[180](1.0,1.0){$x_1$}
				\uput[0](4.0,1.0){$x_2$}
				\uput[0](4.0,4.0){$x_3$}
				\uput[180](1.0,4.0){$x_4$}
				\psline[linestyle=dashed,arrowscale=2]{->}(1.1,1.0)(3.9,1.0)
				\psline[linestyle=dashed,arrowscale=2]{->}(1.1,1.1)(3.9,3.9)
				\psline[linestyle=dashed,arrowscale=2]{<-}(1.0,1.1)(1.0,3.9)
			\end{pspicture}
		\end{pdfpic}
	\end{minipage}
	\caption{Possible difference of communication and dependency graph}
	\label{Setup and Preliminaries:fig:communication and dependency}
\end{figure}

Moreover, as sent information may be delayed or even lost, we want to allow for considering old information on neighbours and variable lengths of this information. Note that the latter may also allow agents to skip recomputations of controls. Introducing $N_q$ as the length and $n_q$ as the time instant at which neighbour $q$ has computed the state sequence $(x_q^{n_q}(0), \ldots, x_q^{n_q}(N_q))$, we define the exchanged neighbouring information as follows:

\begin{definition}\label{Setup and Preliminaries:neighbouring information}
	Suppose that at time instant $n \in \N_0$ agent $p$ knows the state sequences $x_q^{n_q}(\cdot) = (x_q^{n_q}(0), \ldots, x_q^{n_q}(N_q))$, $N_q \in \N_0$, computed at time instant $n_q \leq n$ for a given \textit{neighbouring index set} $\cIp(n)$, that is $q \in \cIp(n)$ with $p \not \in \cIp(n)$. We define the \textit{neighbouring information} as
	\begin{align*}
		\ip(n) = \{ (q, n_q, N_q, x_q^{n_q}(\cdot)) \mid q \in \cIp(n) \}
	\end{align*}
	being an element of the set $\bIp = 2^Q \mbox{ with } Q = ( \cP \setminus \{ p \} ) \times \N_0 \times \N_0 \times \X^{\N}$.
\end{definition}

Knowing the states of neighbouring systems for a certain time period, we can define the index set used within the ``projection'' of the constraint set $\bX$.

\begin{definition}\label{Setup and Preliminaries:neighbouring prediction index set}
	For a given time instant $n \in \N_0$ and an agent $p \in \cP$ with neighbouring information $\ip(n)$, we call the set of systems $q \in \cIp(n) \setminus \{p\}$ which are imposing constraints on system $p$ at time instant $n + k \in \N_0$, $k \geq 0$ \textit{neighbouring prediction index set}. This set is given by
	\begin{align*}
		\cIp(n, k) = \{ q \in \cIp(n) \setminus \{p\} \mid n + k \leq n_q + N_q \}.
	\end{align*}	
\end{definition}

Similar to possible moves of the cars in Example \ref{Setup and Preliminaries:example}, we can use the partial state constraint set connected to neighbouring information available to an agent and define the set of admissible controls from which the control sequence $\up(\cdot)$ can be chosen, cf. Figures \ref{Setup and Preliminaries:fig:carexample0}, \ref{Setup and Preliminaries:fig:carexample1}.

\begin{definition}\label{Setup and Preliminaries:set of admissible control sequences}
	 Given a time instant $n \in \N_0$ and an agent $p \in \cP$ with initial value $\xpzero$ and neighbouring information $\ip(n)$, we define the \textit{set of admissible control sequences} for system $p$ at time instant $n$ as
	\begin{align*}
		\bUpad(n, \xpzero, \ip(n)) = \{ \up(\cdot) \in \Up^{\N_0} \mid & \mbox{ for all $k = 0, 1, \ldots$ we have $\up(k) \in \bUp$ and } \\
		& (\xpu(k, \xpzero), ( x_q^{n_q}(k + n - n_q) )_{\cIp(n, k)}) \in \bX_{\{p\} \cup \cIp(n, k)} \}.
	\end{align*}
\end{definition}

Using an NMPC algorithm is one possibility to compute a control from the set of admissible controls. In particular, the method tries to approximate a control sequence such that the functional
\begin{align}	
	\label{Setup and Preliminaries:cost functional infty}
	\Jpinfty(\xpzero, \up) = \sum\limits_{k = 0}^{\infty} \lp(\xpu(k, \xpzero), \up(k))
\end{align}
is minimized over all admissible control sequences, that is sequences $\up(\cdot)$ with $\up(k) = \ups(0)$ for all $k \in \N_0$ with $\ups \in \bUpad(k, \xpu(k, \xpzero), \ip(k))$. Here, the function $\lp$ is a stage cost function penalizing both the distance of the state to the desired equilibrium and the used control. A popular choice for this function is $\lp(\xp, \up) = \| \xp \|_{\xpref} + \lambda \| \up \|_{\upref}$ with weighting parameter $\lambda > 0$. 

Computing a control minimizing \eqref{Setup and Preliminaries:cost functional infty} is, in general, computational intractable. To circumvent this issue the NMPC algorithm uses the truncated cost functional
\begin{align}	
	\label{Setup and Preliminaries:cost functional p}
	\JpNp(\xpzero, \up) = \sum\limits_{k = 0}^{\Np - 1} \lp(\xpu(k, \xpzero), \up(k))
\end{align}
with finite horizon of length $\Np$ and initial value $\xpzero$. Hence, a finite minimizing control sequence $\ups \in \bUpNpad(n, \xpzero, \ip(n))$ is computed with
\begin{align*}
	\bUpNpad(n, \xpzero, \ip(n)) = \{ \up(\cdot) \in \UpNp \mid & \mbox{ for all $k = 0, \ldots, \Np$ we have $\up(k) \in \bUp$ and } \\
	& (\xpu(k, \xpzero), ( x_q^{n_q}(k + n - n_q) )_{\cIp(n, k)}) \in \bX_{\{p\} \cup \cIp(n, k)} \}.
\end{align*}
In the following we assume that a minimizing control sequence exists and denote the corresponding optimal value function by
\begin{align*}
	\VpNp(\xp(n), \ip(n)) = \min_{\up \in \bUpNpad(n, \xp(n), \ip(n))} \JpNp(\xp(n), \up)
\end{align*}
where the minimizing control sequence is given by
\begin{align*} 
	\ups = \argmin_{\up \in \bUpNpad(n, \xp(n), \ip(n))} \JpNp(\xp(n), \up).
\end{align*}
Here, the $\argmin$ operator is used in the following sense: given a map $a:\U \to \R$, a nonempty subset $\widetilde{\U} \subseteq \U$ and a value $\us \in \widetilde{\U}$ we write $\us = \argmin_{u \in \widetilde{\U}} a(u)$ if and only if $a(\us) = \min_{u \in \widetilde{\U}} a(u)$ holds. Note that we do not require uniqueness of the minimizer $\us$. In case of uniqueness the $\argmin$ operator can be understood as an assignment, otherwise it is just a convenient way of writing ``$\us$ minimizes $a(u)$''.

Having obtained a minimizing sequence $\ups(\cdot)$, only the first element $\ups(0)$ of the control sequence is implemented. Then the entire problem is shifted forward in time by one time instant and both a new initial value and neighbouring information need to be obtained. Applying this method iteratively results in a feedback law which assigns the first element of the minimizing control sequence $\ups(\cdot)$ to the current state of the $p$-th system $\xp(n)$ and the neighbouring information $\ip(n)$ of the corresponding agent, i.e. a map
\begin{align}
	\label{Setup and Preliminaries:eq:feedback}
	\mupNp: (\xp(n), \ip(n)) \mapsto \ups(0).
\end{align}
Accordingly, the closed loop solution of the $p$-th system is given by
\begin{align}
	\label{Setup and Preliminaries:eq:closed loop solution p}
	\xp(n + 1) = f(\xp(n), \mupNp(\xp(n), \ip(n))) \quad \text{with} \quad \xp(0) = \xpzero.
\end{align}

Using this setting, we first show conditions which guarantee asymptotic stability of the closed loop for local controllers. 

\section{Stability}
\label{Section:Stability}

While commonly endpoint constraints or a Lyapunov function type endpoint weight are used to ensure stability of the closed loop, see, e.g., the articles of \citet{KG1988}, \citet{ChAl1998}, \citet{JH2005} and \citet{Graichen2010}, we consider the plain NMPC version without these modifications. In order to guarantee stability in this case, we use the ``relaxed'' version of the dynamic programming principle, cf. \citet{LR2006}. In particular, one can show asymptotic stability of \eqref{Setup and Preliminaries:system} in a trajectory based setting using a relaxed Lyapunov condition, see \citet[Proposition 7.6]{GP2011}. Note that this stability result requires a centralized setting and the horizons to satisfy $N_p = N$ for all $p \in \cP$. Hence, for the overall system \eqref{Setup and Preliminaries:system} we denote the combined stage costs by $\l(x(n), u(n))$, the finite and infinite cost functional by $\JN(\xzero, u)$, $\Jinfty(\xzero, u)$ and the corresponding combined value functions by $\VN(x(n))$, $\Vinfty(x(n))$ which allows us to apply the stability result of \citet[Proposition 7.6]{GP2011}:

\begin{proposition}\label{Stability:proposition:online alpha aposteriori}
	Consider a feedback law $\muN: \bX \rightarrow \bU$ and the closed loop trajectory $x(\cdot)$ of \eqref{Setup and Preliminaries:system} with control $u = \muN$ and initial values $x(0) \in \bX$ to be given. If the optimal value function $\VN:\bX \rightarrow \R_{\geq 0}$ satisfies 
	\begin{align}
		\label{Stability:proposition:online alpha aposteriori:eq1}
		\VN(x(n)) \geq \VN(f(x(n), \muN(x(n))) + \alpha \l(x(n), \muN(x(n))) 
	\end{align}
	for some $\alpha \in (0, 1]$ and all $n \in \N_0$, then 
	\begin{align}
		\label{Stability:proposition:online alpha aposteriori:eq2}
		\alpha \Vinfty(x(n)) \leq \alpha \Jinfty(x(n), \muN) \leq \VN(x(n)) \leq \Vinfty(x(n)) 
	\end{align}
	holds for all $n \in \N_0$.

	If, in addition, there exist $\alpha_1,\alpha_2,\alpha_3 \in \cK_\infty$ such that
	\begin{align}
		\label{Stability:proposition:online alpha aposteriori:eq3}
		\alpha_1(\| x \|_{\xref}) \leq \VN(x) \leq \alpha_2(\| x \|_{\xref}) \quad \mbox{ and } \quad \l(x, u) \ge \alpha_3(\| x \|_{\xref})
	\end{align}
	holds for all $x(n) \in \bX$ with $n \in \N_0$, then there exists a function $\beta \in \cKL$ which only depends on $\alpha_1, \alpha_2, \alpha_3$ and $\alpha$ such that the inequality
	\begin{align}
		\label{Stability:proposition:online alpha aposteriori:eq4}
		\| x(n) \|_{\xref} \leq \beta(\| x(0) \|_{\xref},n)
	\end{align}
	holds for all $n \in \N_0$, i.e., $x$ behaves like a trajectory of an asymptotically stable system. 
\end{proposition}

The key assumption in Proposition \ref{Stability:proposition:online alpha aposteriori} is the relaxed Lyapunov--inequality \eqref{Stability:proposition:online alpha aposteriori:eq1} in which $\alpha$ can be interpreted as a lower bound for the rate of convergence. From the literature, it is well--known that this condition is satisfied for sufficiently long horizons $N$, cf. \citet{JH2005}, \citet{GMTT2005} or \citet{AB1995}, and that a suitable $N$ may be computed via methods described in \citet[Chapter 7]{GP2011} or \citet{G2010}.

Now we consider a distributed setting of Proposition \ref{Stability:proposition:online alpha aposteriori} using compositions to combine of the set of systems \eqref{Setup and Preliminaries:system p}. The idea of such compositions is to introduce a weighting among the subsystems which in our further analysis will allow for increases of costs along the closed loop for some subsystems.

\begin{proposition}\label{Stability:proposition:online alpha aposteriori distributed}
	Consider feedback laws $\mupN: \bXp \times \bIp \rightarrow \bUp$ and closed loop trajectories $\xp(\cdot)$ of \eqref{Setup and Preliminaries:eq:closed loop solution p} with initial values $\xp(0) \in \bXp$ to be given. If the optimal value functions $\VpN:\bXp \rightarrow \R_{\geq 0}$ satisfy
	\begin{align}
		\label{Stability:proposition:online alpha aposteriori distributed:eq1}
		\VpN(\xp(n)) \geq \VpN(\fp(\xp(n), \mupN(\xp(n), \ip(n)))) + \alpha \lp(\xp(n), \mupN(\xp(n), \ip(n))) 
	\end{align}
	for some $\alpha \in (0, 1]$ and all $n \in \N_0$, then for any weighting function 
	$\gamma: \R^P \to \R_{\geq 0}$, $\gamma(x)=\gamma^\top x$, $\gamma_i \in \R_{> 0}$ we have that \eqref{Stability:proposition:online alpha aposteriori:eq2} holds for all $n \in \N_0$ with
	\begin{align*}
		\VN(x) := \gamma( ( V_1^N(x_1), \ldots, V_P^N(x_p))^\top ) \quad \mbox{and} \quad \l(x, u) := \gamma( ( \l_1(x_1, u_1), \ldots, \l_P(x_P, u_P))^\top ).
	\end{align*}

	If, in addition, for every $p \in \cP$ there exist $\alpha_1^p,\alpha_2^p,\alpha_3^p \in \cK_\infty$ such that
	\begin{align}
		\label{Stability:proposition:online alpha aposteriori distributed:eq2}
		\alpha_1^p(\| \xp \|_{\xpref}) \leq \VpN(x) \leq \alpha_2^p(\| \xp \|_{\xpref}) \quad \mbox{ and } \quad \lp(\xp, \up) \ge \alpha_3^p(\| \xp \|_{\xpref})
	\end{align}
	holds for all $\xp(n) \in \bX$ with $n \in \N_0$, then there exists a function $\beta \in \cKL$ which only depends on $\gamma$, $\alpha$ and all $\alpha_1^p, \alpha_2^p, \alpha_3^p$, $p \in \cP$, such that \eqref{Stability:proposition:online alpha aposteriori:eq4} holds for all $n \in \N_0$. 
\end{proposition}

\begin{proof}
	Defining the abbreviations $V_\cP^N(x(n)) := ( V_1^N(x_1(n)), \ldots, V_P^N(x_P(n)))^\top$ and $\l_\cP(x(n), \muN(x(n), \i_\cP(n))) := ( \l_1(x_1(n), \mu_1^N(x_1(n), \i_1(n))), \ldots, \l_P(x_P(n), \mu_P^N(x_P(n), \i_P(n))) )^\top$ we combine all inequalities \eqref{Stability:proposition:online alpha aposteriori distributed:eq1} for $p \in \cP$ and obtain
	\begin{align}
		\label{Stability:proposition:online alpha aposteriori distributed:proof:eq1}
		\gamma( V_\cP^N(x(n)) ) \geq \gamma ( V_\cP^N(x(n + 1))) + \alpha \gamma ( \l_\cP(x(n), \muN(x(n), \i_\cP(n))) ) .
	\end{align}
	using linearity of $\gamma$. Now we use the definition of $\VN$ and $\l$ which gives us \eqref{Stability:proposition:online alpha aposteriori:eq1}. Hence,  \eqref{Stability:proposition:online alpha aposteriori:eq2} follows directly from Proposition \ref{Stability:proposition:online alpha aposteriori}. Similarly, \eqref{Stability:proposition:online alpha aposteriori:eq4} follows by definition of $\VN$ and $\l$ which together with $\alpha_i(r) := \gamma ( (\alpha_i^1(r), \ldots, \alpha_i^P(r))^\top )$, $i = 1, 2, 3$, and again Proposition \ref{Stability:proposition:online alpha aposteriori} shows the assertion.
\end{proof}

Certainly, condition \eqref{Stability:proposition:online alpha aposteriori distributed:eq1} would be desireable since it guarantees a decrease in $\VpN$ for each $p \in \cP$. In practice, however, one would usually expect $\VpN$ to decrease for some $p \in \cP$ while it increases for others as shown in the following example:

\begin{example}\label{Stability:example}
	Consider the setting of Example \ref{Setup and Preliminaries:example} where we suppose that each agent $p = 1, 2$ is optimizing using its running costs $\lp(\xp, \up) = \| \xp - \xpref \|_2^2$ with $\xref_1 = ( -2, 0)^\top$ and $\xref_2 = (2, 0)^\top$ and complete neighbouring information for initial conditions $\xzero_{1} = (1, 0)^\top$, $\xzero_{2}=(-1, 0)^\top$.\\
	Due to the constraints \eqref{Setup and Preliminaries:example:collision constraint}--\eqref{Setup and Preliminaries:example:passing constraint} one car has to wait before entering the bridge, cf. Figure \ref{Setup and Preliminaries:fig:carexample2}, and even has to move aside as shown in Figure \ref{Stability:fig:carexampleoptimal1}. Without loss of generality we assume that system $p = 2$ moves aside. 
	\begin{figure}[!ht]
		\begin{center}
			\subfloat[Optimal solution for $n = 1$ and $N = 1$\label{Stability:fig:carexampleoptimal1}]{\includegraphics[width=0.32\textwidth]{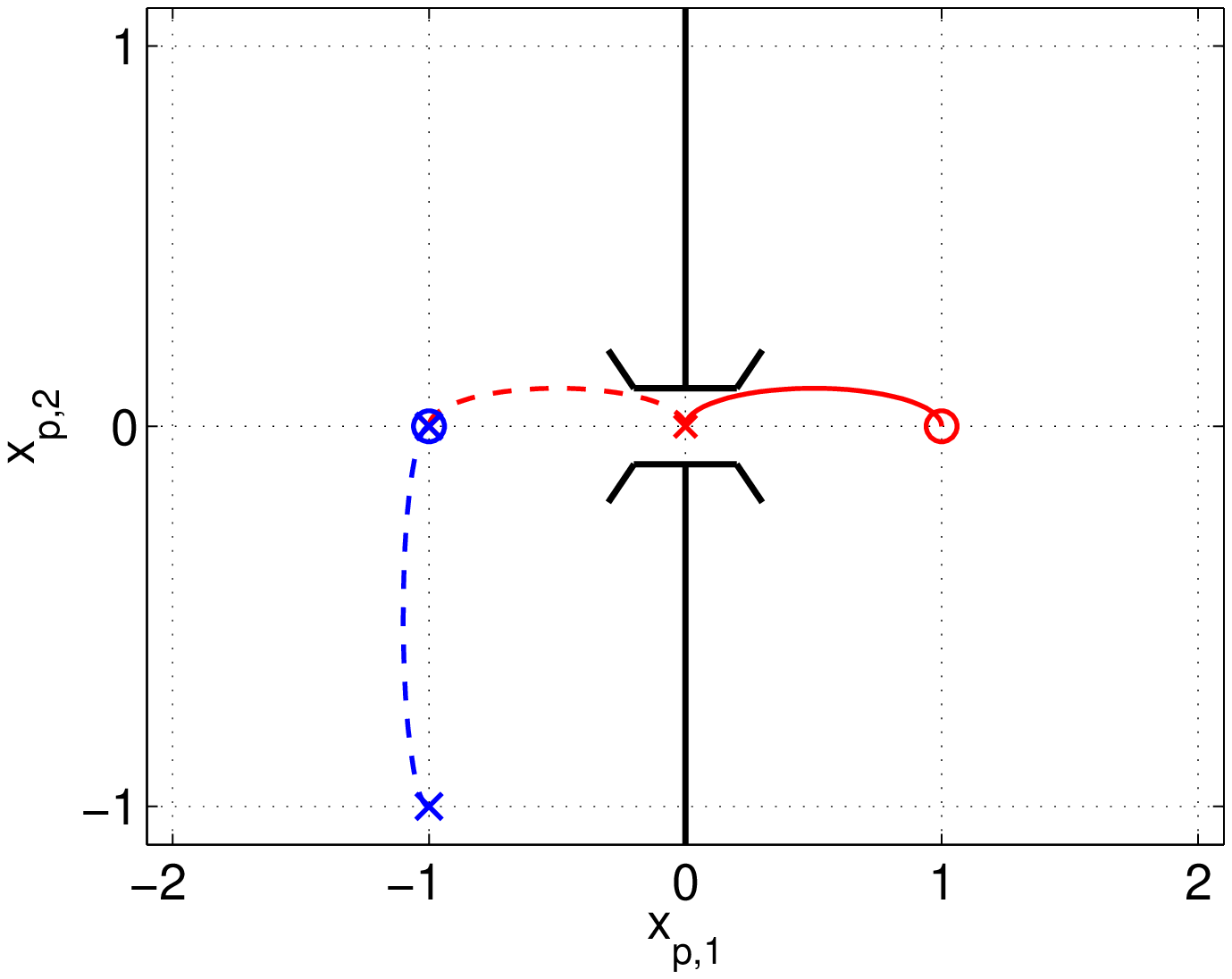}}
			\subfloat[Optimal solution for $n = 0$ and $N = 2$\label{Stability:fig:carexampleoptimal2}]{\includegraphics[width=0.32\textwidth]{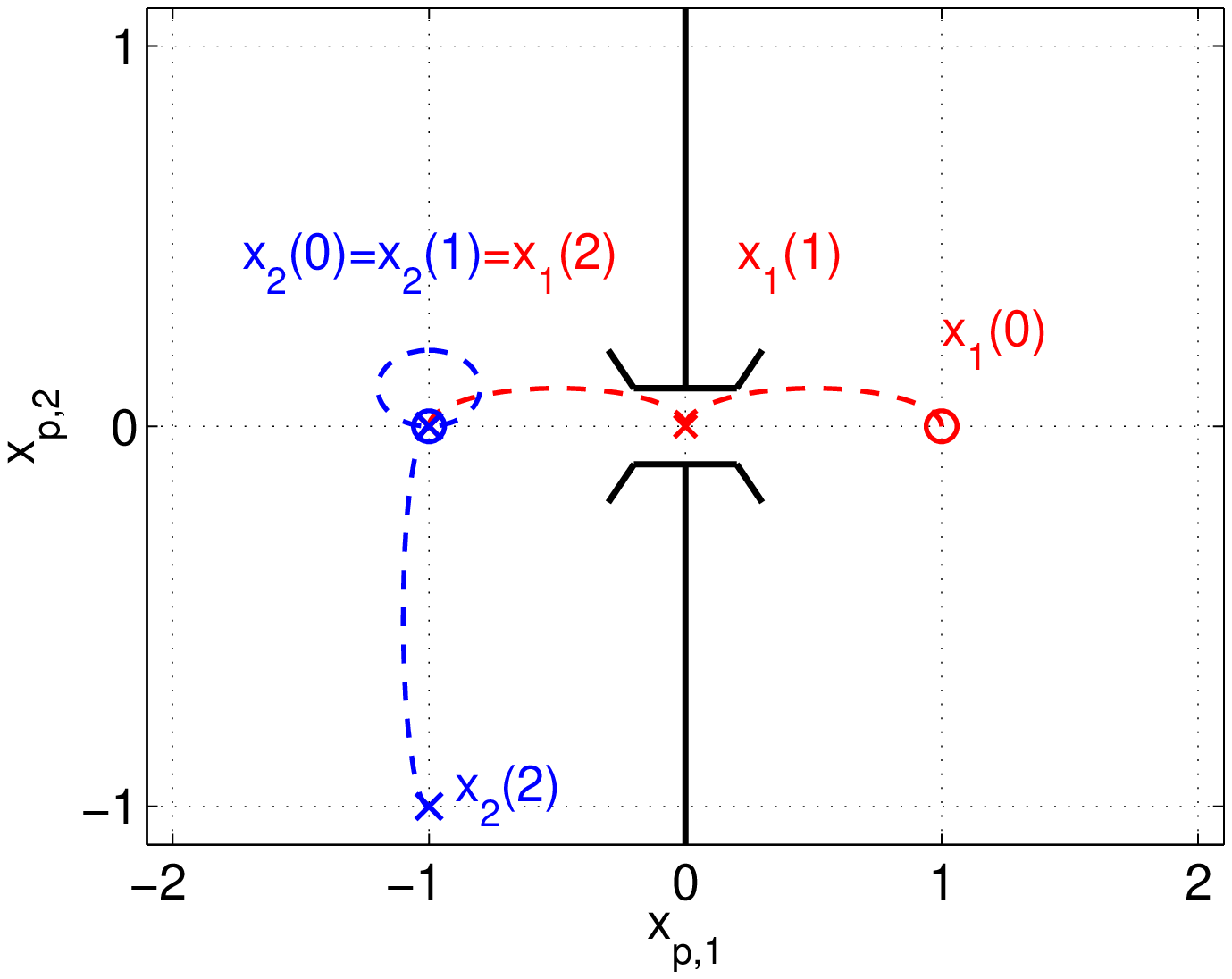}}
			\subfloat[Global optimal solution\label{Stability:fig:carexampleoptimal3}]{\includegraphics[width=0.32\textwidth]{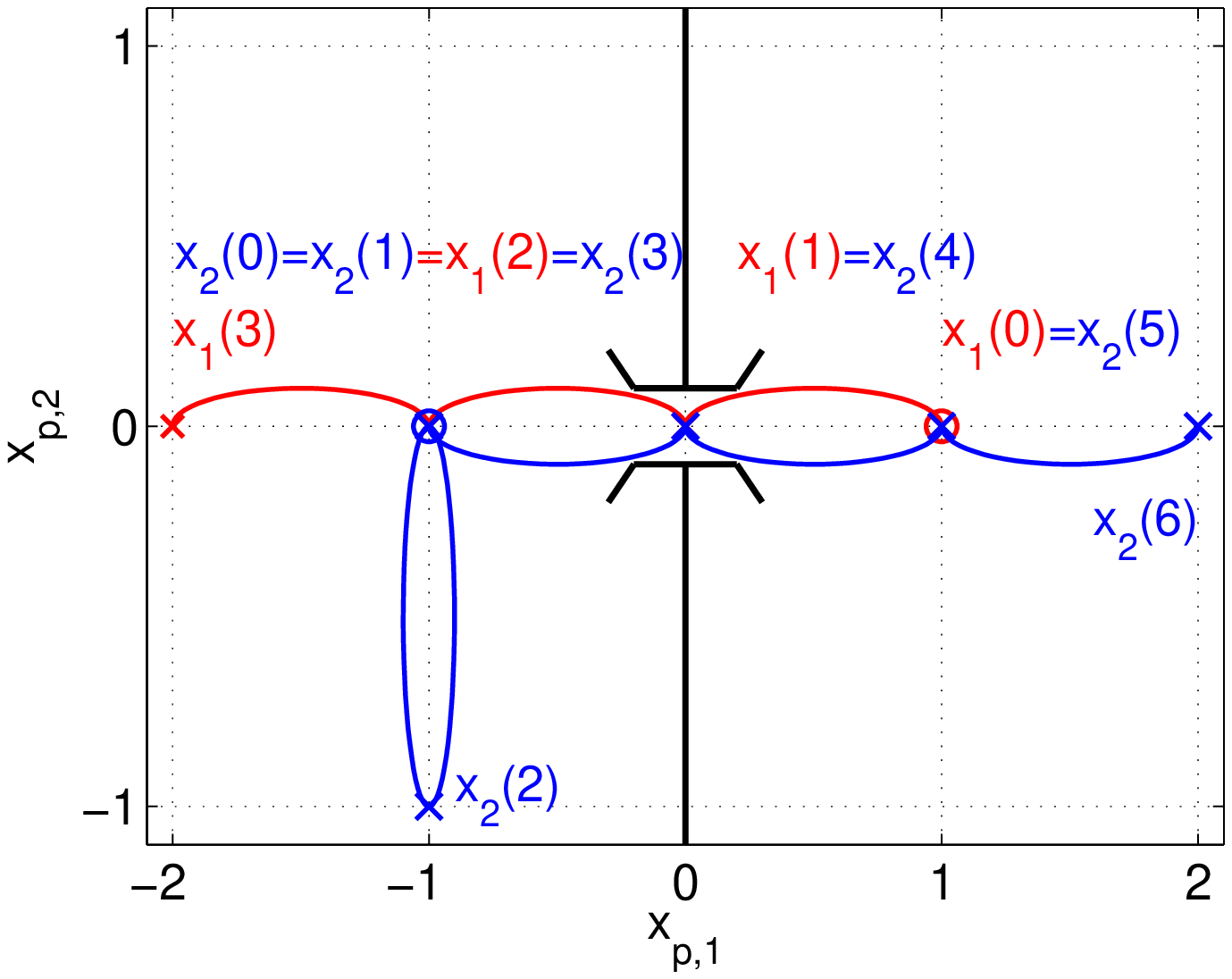}}
		\end{center}
		\caption{Optimal open and closed loop trajectories}\label{Stabiliy:fig:carexampleoptimal}
	\end{figure}		
	Although the optimal control for agent $p=2$ for any $N \geq 2$ is identical to the global optimal control, cf. Figures \ref{Stability:fig:carexampleoptimal2} and \ref{Stability:fig:carexampleoptimal3}, we obtain $\VN_2(x_2(0), \i_2(0)) \leq \VN_2(x_2(1), \i_2(1))$ for $N = 2$ and $N = 3$. Hence, we cannot guarantee \eqref{Stability:proposition:online alpha aposteriori distributed:eq1} to hold for these $N$ although the closed loop is stable. For larger values of $N$, however, we obtain
	\begin{align*}
		\VN_2(x_2(0), \i_2(0)) & = \begin{cases} 37 & \mbox{if } N = 4 \\ 41 & \mbox{if } N = 5 \\ 42 & \mbox{if } N \geq 6 \end{cases}
		&\quad \text{and} \quad 
		\l_2(x_2(0), \muN_2(x_2(0),\i_2(0))) & = 9, \displaybreak[0] \\
		\VN_2(x_2(1), \i_2(1)) & = \begin{cases} 32 & \mbox{if } N = 4 \\ 33 & \mbox{if } N \geq 5 \end{cases}
		&\quad \text{and} \quad 
		\l_2(x_2(1), \muN_2(x_2(1),\i_2(1))) & = 9,
	\end{align*}
	\begin{align*}
		\left( \VN_2(x_2(n), \i_2(n)) \right)_{n=2,\ldots,5, N \geq 4} & = ( 24, 10, 5, 1 )\quad \text{and} \\
		\left( \l_2(x_2(n), \muN_2(x_2(n),\i_2(n))) \right)_{n=2,\ldots,5, N \geq 4} & = ( 10, 9, 4, 1 )
	\end{align*}
	and $\VN_2(x_2(n), \i_2(n)) = \l_2(x_2(n), \muN_2(x_2(n),\i_2(n))) = 0$ if $N \geq 4$ and $n \geq 6$. Accordingly, the largest values $\alpha$ such that \eqref{Stability:proposition:online alpha aposteriori distributed:eq1} holds are $\alpha = 5/9$ if $N = 4$, $\alpha = 8/9$ if $N = 5$ and $\alpha = 1$ if $N \geq 6$. Hence, we can use Proposition \ref{Stability:proposition:online alpha aposteriori distributed} to conclude asymptotic stability if $N \geq 4$.
\end{example}

In an indeep analysis, \citet[Theorem 5.3]{GW2010} have shown conditions such that for the serial case using the algorithm of Richards and How \cite{RH2004, RH2007} inequalities \eqref{Stability:proposition:online alpha aposteriori distributed:eq1} and \eqref{Stability:proposition:online alpha aposteriori distributed:eq2} hold. Note that although the setting within the articles of \citet{RH2004, RH2007} is for one based on linear dynamics and secondly explicitely includes perturbations in the models, this algorithm can also be used in a nonlinear setting, cf. \citet[Proposition 3.2]{GW2010}. 

Taking a closer look at the proof of Proposition \ref{Stability:proposition:online alpha aposteriori distributed} we see that in fact conditions \eqref{Stability:proposition:online alpha aposteriori distributed:eq1} are only required to guarantee \eqref{Stability:proposition:online alpha aposteriori distributed:proof:eq1} to hold. Note that while condition \eqref{Stability:proposition:online alpha aposteriori distributed:eq1} requires a decrease in $\VpN$ for each $p \in \cP$, in \eqref{Stability:proposition:online alpha aposteriori distributed:proof:eq1} it suffices that $\VpN$, $p \in \cP$ is decreasing under a map $\gamma$. Moreover, we only require $\gamma \in \cK_\infty^P$ in the remainder of the proof of Proposition \ref{Stability:proposition:online alpha aposteriori distributed} to guarantee $\alpha_i \in \cK_\infty$, $i=1, 2, 3$. Accordingly, we obtain the following more general result:

\begin{proposition}\label{Stability:proposition:online alpha aposteriori distributed2}
	Consider feedback laws $\mupN: \bXp \times \bIp \rightarrow \bUp$ and closed loop trajectories $\xp(\cdot)$ of \eqref{Setup and Preliminaries:eq:closed loop solution p} with initial values $\xp(0) \in \bXp$ to be given. If the optimal value functions $\VpN:\bXp \rightarrow \R_{\geq 0}$ satisfy \eqref{Stability:proposition:online alpha aposteriori distributed:proof:eq1} for some $\alpha \in (0, 1]$, $\gamma \in \cK_\infty^P$ and all $n \in \N_0$, then \eqref{Stability:proposition:online alpha aposteriori:eq2} holds for all $n \in \N_0$ with $\VN$ and $\l$ defined as in Proposition \ref{Stability:proposition:online alpha aposteriori distributed}.

	If, in addition, for every $p \in \cP$ there exist $\alpha_1^p,\alpha_2^p,\alpha_3^p \in \cK_\infty$ such that \eqref{Stability:proposition:online alpha aposteriori distributed:eq2} holds for all $\xp(n) \in \bX$ with $n \in \N_0$, then there exists a function $\beta \in \cKL$ which only depends on $\gamma$, $\alpha$ and all $\alpha_1^p, \alpha_2^p, \alpha_3^p$, $p \in \cP$, such that \eqref{Stability:proposition:online alpha aposteriori:eq4} holds for all $n \in \N_0$. 
\end{proposition}

\begin{proof}
	Follows directly from the proof of Proposition \ref{Stability:proposition:online alpha aposteriori distributed}.
\end{proof}

The conclusion that can be drawn from Proposition \ref{Stability:proposition:online alpha aposteriori distributed2} is that the weighting function $\gamma$ may allow us to partially violate condition \eqref{Stability:proposition:online alpha aposteriori distributed:eq1}. Since \eqref{Stability:proposition:online alpha aposteriori distributed:eq1} is typically fulfilled if the horizon $\Np$ is large enough, a good choice of $\gamma$ may reduce the horizon length $\Np$ as we will see in the following example:

\begin{example}\label{Stability:example2}
	Consider Example \ref{Stability:example} and suppose $\gamma$ to be the 1--norm, then we obtain for $N = 2$
	\begin{align*}
		\left( \sum\limits_{p = 1}^2 \VpN(\xp(n), \ip(n)) \right)_{n=0, \ldots, 5} = ( 31, 24, 20, 13, 5, 1 ) \\
		\left( \sum\limits_{p = 1}^2 \lp(\xp(n), \mupN(\xp(n),\ip(n))) \right)_{n=0, \ldots, 5} = ( 18, 13, 11, 9, 4, 1 )
	\end{align*}
	and $\sum_{p = 1}^2 \VpN(\xp(n), \ip(n)) = \sum_{p = 1}^2 \lp(\xp(n), \mupN(\xp(n),\ip(n))) = 0$ for $n \geq 6$. Hence, \eqref{Stability:proposition:online alpha aposteriori distributed:proof:eq1} holds with $\alpha = 4/13$ and we obtain asymptotic stability of the closed loop by Proposition \ref{Stability:proposition:online alpha aposteriori distributed2}. Since $\alpha > 0$ holds for all $N \geq 2$ this example illustrates the advantage of considering condition \eqref{Stability:proposition:online alpha aposteriori distributed:proof:eq1} instead of \eqref{Stability:proposition:online alpha aposteriori distributed:eq1}.
\end{example}

As outlined before Proposition \ref{Stability:proposition:online alpha aposteriori distributed2}, under certain conditions the algorithm of \citet{RH2004, RH2007} can be applied to generate solutions such that \eqref{Stability:proposition:online alpha aposteriori distributed:eq1} and \eqref{Stability:proposition:online alpha aposteriori distributed:eq2} hold. However, the nature of this algorithm is serial, that is while one agent $p \in \cP$ is computing its control, all other agents $q \in \cP \setminus \{ p \}$ have to wait until agent $p$ finished computing. Hence, if the number of systems $P$ is large, such an algorithm may cause rather long waiting times, a feature which may be unwanted if fast sampling is used. Still, as noted in \citet[Section 7]{RH2007}, due to its decentralized nature the dimension of each problem is significantly smaller and hence the algorithm reduces the numerical effort compared to a centralized solution considerably. \\
Apart from the serial nature, the algorithm of Richards and How requires accessibility to the full neighbouring information, i.e. a full communication graph. Additionally, an agent $p \in \cP$ always uses the latest available neighbouring information to compute a minimizing control $\ups$ which results in a full dependency graph. While the latter condition on the dependency graph may be relaxed easily, it is a complex task to obtain a parallel algorithm and to relax the requirement of a full communication graph.

\section{The covering algorithm}
\label{Section:The covering algorithm}

In this section we provide a covering algorithm which is a modification of the Algorithm of \citet{RH2004, RH2007} and allows us to run the agents $p \in \cP$ in parallel if they are independent from one another. Unfortunately, working in a parallel distributed setting omits the use of standard techniques from optimization such as first and second order information of the cost functional and the constraints for the interlink between systems to search for optimal controls.

To circumvent this deficiency we introduce abstract maps $\Pi, \Theta: 2^\cP \to 2^\cP$ which denote priority and deordering rules, see \citet{H1989} and \citet{Ba1998}. The aim of this section is to show how much parallelism can be expected using the algorithm we propose next and the basic properties of $\Pi$, $\Theta$ being a permutation and a self concatenation mapping respectively. \\
The structural layout of the algorithm we present now is closely related to the NMPC algorithm outlined in Section \ref{Section:Setup and Preliminaries}:

\begin{algorithm}\label{The covering algorithm:alg:abstract}
	Set lists $\cP_1: = ( 1, \ldots, P )$ and $\cP_p := \emptyset$ for $p = 2, \ldots, P$, $n := 0$ and $\ip(n) := \emptyset$ for $p = 1, \ldots, P$.
	\begin{enumerate}
		\item[1.] Obtain new measurements $\xp(n)$ for $p \in \cP$.
		\item[2a.] \textit{(Decision memory and deordering rule)} For $i$ from $2$ to $P$ do
		\begin{enumerate}
			\item[] For $j$ from $1$ to $\sharp \cP_{i}$ do
			\begin{enumerate}
				\item[(i)] Set $\cI_{p_j}(n) := \Theta(\cI_{p_j}(n)) \subsetneq \cI_{p_j}(n)$
				\item[(ii)] If $\cI_{p_j}(n) = \emptyset$, then remove $p_j$ from $\cP_{i}$ and set $\cP_{1} := (\cP_{1}, p_j)$\\
				Else: If $\tilde{m} = \min_{k \in \cP_{m}, p_k \in \cI_{p_j}(n)} m < i$ holds, then remove $p_j$ from $\cP_{i}$ and set $\cP_{\tilde{m}} := (\cP_{\tilde{m}}, p_j)$
			\end{enumerate}
		\end{enumerate}
		\item[2b.] Compute a control $\ups(\cdot)$ minimizing \eqref{Setup and Preliminaries:cost functional infty} or \eqref{Setup and Preliminaries:cost functional p} with $\xpzero = \xp(n)$ and send information to all agents $q \in \{ q \in \cP \mid q \in \cP_j, p \in \cP_i \mbox{ and } j \geq i \}$ for $p \in \cP$ in parallel
		\item[2c.] \textit{(Priority rule)} For $i$ from $1$ to $P$ do
		\begin{enumerate}
			\item[(i)] If $\sharp \cP_{i} \in \{0, 1\}$, goto Step 3. \\
			Else: Sort index list by  setting $\cP_{i} := \Pi(\cP_{i})$
			\item[(ii)] For $j$ from $2$ to $\sharp \cP_{i}$ do
			\begin{enumerate}
				\item[] If system $p_{j}$ violates constraints imposed by systems $p_{k}$, $k < j$, then set $\cP_{i+1} := (\cP_{i+1} , j )$ and $\cI_{j}(n) := \cI_{j}(n) \cup \{ p_{k} \in \cP_{i} \setminus \cP_{i+1} \mid \text{$p_{k}$, $k < j$, induces constraints violated by system $p_{j}$}\}$ and set $\cP_{i} := \cP_{i} \setminus \cP_{i+1}$
			\end{enumerate}
			\item[(iii)] Compute a control $\ups(\cdot, \ip(n))$ minimizing \eqref{Setup and Preliminaries:cost functional infty} or \eqref{Setup and Preliminaries:cost functional p} for all $p \in \cP_{i+1}$ in parallel and send information to all agents $q \in \{ q \in \cP \mid q \in \cP_j, j \geq i \}$
		\end{enumerate}
		\item[3.] Implement $\mupNp(\xp(n), \ip(n)) := \ups(0)$, set $n := n + 1$ and goto Step 1.
	\end{enumerate}
\end{algorithm}
The general idea of the algorithm is to first generate priority lists $\cP_i$ of the systems according to the rule $\Pi$ and according to their interconnection with other systems, cf. Step 2c, -- just as the right-before-left rule in street traffic or the search direction in optimization methods. Secondly, these lists are used to remember earlier decisions which avoids generating periodic behaviour. This part of the algorithm, contained in Step 2a, is inspired by Bland's rule and the lexicographic ordering method used in the simplex algorithm to cope with degeneracy. Last, the deordering rule $\Theta$ which is used together with the memory in Step 2a offers a possibility to break up earlier decisions. Proceeding this way avoids blockages and reduces both the number of priority lists and thereby the numerical effort to compute the control sequences.

We like to mention that Algorithm \ref{The covering algorithm:alg:abstract} can be extended to an iterative computation of the controls $\ups$, $p \in \cP$. To this end only a few steps within the optimization method used to solve the problems of Steps 2b and 2c(iii) are performed. Additionally a second loop containing Steps 2b and 2c is introduced which is terminated if some stopping criterion like the suboptimality based criterion given in \citet{GP2010} is satisfied. Note that the algorithm also allows us to stop agents during such an iterative computation, i.e. if \eqref{Stability:proposition:online alpha aposteriori distributed:eq1} is satisfied for some $\alpha \geq \overline{\alpha} \in (0, 1)$. Since we allowed for using old and even outdated information in Definition \ref{Setup and Preliminaries:neighbouring information}, the algorithm even allows to block any computations of some agents for a certain period depending on the length of an agents prediction without compromising feasibility.

Given Algorithm \ref{The covering algorithm:alg:abstract}, we first consider the question whether a feasible feedback $\mupNp$ can be computed via Algorithm \ref{The covering algorithm:alg:abstract}:
\begin{theorem}\label{The covering algorithm:thm:feasibility}
	Assume a feasible initial value $x_0 \in \bX$ for system \eqref{Setup and Preliminaries:system} to be given. Suppose that for all $p \in \cP$ and all $n \in \N_0$ we have that the sets of admissible controls $\bUpad(n, \xp(n), \ip(n))$ in case of cost functional \eqref{Setup and Preliminaries:cost functional infty} or $\bUpNpad(n, \xp(n), \ip(n))$ in case of cost functional \eqref{Setup and Preliminaries:cost functional p} in Steps 2b and 2c(iii) are not empty, then the closed loop solutions \eqref{Setup and Preliminaries:eq:closed loop solution p} satisfy $x(n) = ( x_1(n)^\top, \ldots, x_P(n)^\top )^\top \in \bX$.
\end{theorem}
\begin{proof}
	Using $x_0 \in \bX$ and $\bUpad(0, \xp(0), \ip(0)) \not = \emptyset$ for all $p \in \cP$ in case of cost functional \eqref{Setup and Preliminaries:cost functional infty} or $\bUpNpad(0, \xp(0), \ip(0)) \not = \emptyset$ for all $p \in \cP$ in case of cost functional \eqref{Setup and Preliminaries:cost functional p}, we obtain from Steps 2b and 2c(iii) that optimal controls $\ups(\cdot, \ip(0))$ exist for all $p \in \cP$. Hence, by definition of the closed loop in \eqref{Setup and Preliminaries:eq:closed loop solution p} and Step 3 we obtain that $x(1) = ( x_1(1)^\top, \ldots, x_P(1)^\top )^\top \in \bX$ holds. Applying the same argumentation inductively for all $n \in \N_0$ the assertion follows.
\end{proof}

Before showing results for Algorithm \ref{The covering algorithm:alg:abstract} together with general priority and deordering rules $\Pi$, $\Theta$, we like to illustrate both rules using the example outlined in Section \ref{Section:Stability}. The idea of the priority rule is straight forward. In fact, we have already used it in Example \ref{Stability:example} to solve the blockage in the very first step:

\begin{example}
	Again consider Example \ref{Stability:example}. Due to the constraint sets $\bX$, $\bU$ and the dynamics of the systems $\fp$, one of the agents $p$ has to move aside first to let the system of the other agent pass by before it can proceed towards its desired equilibrium. Putting priority of agent $p = 1$ into a mathematical form, we see that $\Pi$ can be implemented as a lexicographic ordering, that is a list $\cL$ is mapped to its minimal permutation with respect to the dictionary ordering $<^d$ induced by the total orderings $\{ <_1, \ldots, <_m \}$ where $m$ is the length of the list $\cL$ and $<_i$, $i = 1, \ldots m$ is the usual ordering $<$ of the natural numbers $\N$.
\end{example}

Apart from the lexicographic ordering, also other heuristics like the greedy heuristic might be used.
It is not clear how the priority rule should be chosen in a nonlinear setting, and throughout this work we will not focus on this question but instead concentrate on general properties of Algorithm \ref{The covering algorithm:alg:abstract}. \\
The idea of the deordering rule $\Theta$ is more involved as it may interfere with the idea of keeping track of earlier decisions. The purpose of this rule is to reduce the number of the priority lists since Step 2c of Algorithm \ref{The covering algorithm:alg:abstract} is a serial call for all lists $\cP_i$. Accordingly, agents $p \in \cP_{i+1}$ always have to wait until all agents $p \in \cP_i$ have finished computing, a fact we wish to avoid. Note that this serial nature is independent from the parallel computation of control sequences $\ups$, $p \in \cP_i$. Using the deordering rule $\Theta$ allows us to ``test'' whether a system $p \in \cP_i$ still interferes with all systems $p \in \cP_k$, $k < i$, or if it can be inserted into a different priority list $\cP_k$, $k < i$, causing the number of lists and hence the number of non parallel steps to shrink. Yet even if system $p$ cannot be inserted in a different priority list, applying the deordering rule might still result in reducing the size of the neighbouring index set $\cIp(n)$. If this is the case, then the number of constraints of system $p$ is reduced which in turn reduces the numerical effort to compute the control sequence $\ups$.

\begin{example}
	Consider once more Example \ref{Stability:example} with $\Theta(\cP) = \emptyset$. Applying Algorithm \ref{The covering algorithm:alg:abstract} we obtain that $\us_2$ depends on the solution of system $p = 1$ for $n \in \{ 0, 1\}$ only whereas for all $n \geq 2$ both problems can be solved in parallel.
\end{example}

Turning towards the central point of this section, we now analyze how much parallelism is possible even if we do not know the exact sorting and testing operators $\Pi, \Theta$. Based on conditions on the priority lists $\cP_i$ our first result shows in which case all agents can compute their controls independently from each other:

\begin{lemma}\label{The covering algorithm:lem:independence}
	Suppose that for given systems \eqref{Setup and Preliminaries:system p}, maps $\Pi, \Theta: 2^\cP \to 2^\cP$ and $n \in \N_0$ we have that $\cP_2 = \emptyset$ holds in Step 2c(i) of Algorithm \ref{The covering algorithm:alg:abstract}. Then every agent $p \in \cP$ can compute its control sequence independently of all other agents $q \in \cP \setminus \{ p \}$.
\end{lemma}
\begin{proof}
	Since $\cP_2 = \emptyset$ Step 2c(ii) of Algorithm \ref{The covering algorithm:alg:abstract} guarantees that there are no systems $p_1, p_2 \in \cP_1$, $p_1 \not = p_2$, such that $p_1$ induces a constraint on $p_2$ which is violated by $p_2$, i.e. $\cIp(n) = \emptyset$ for all $p \in \cP$. Hence, for each agent $p \in \cP$ the set of admissible controls simplifies to
	\begin{align*}
		\bUpad(n, \xpzero, \ip(n)) = \{ \up(\cdot) \in \Up^{\N_0} \mid & \mbox{ $\up(k) \in \bUp$ and $\xpu(k, \xpzero) \in \bXp$ for all $k \in \N_0$} \}
	\end{align*}
	if cost functional \eqref{Setup and Preliminaries:cost functional infty} or
	\begin{align*}
		\bUpNpad(n, \xpzero, \ip(n)) = \{ \up(\cdot) \in \Up^{\Np} \mid & \mbox{ $\up(k) \in \bUp$ and $\xpu(k, \xpzero) \in \bXp$} \\
		& \mbox{for all $k \in \{0, \ldots, \Np\}$} \}
	\end{align*}
	if cost functional \eqref{Setup and Preliminaries:cost functional p} is considered with $\xpzero = \xp(n)$ showing the assertion.
\end{proof}

Using the self-concatenation property of the map $\Theta$, we can also show that under certain conditions the priority lists show dependency of agents:

\begin{lemma}\label{The covering algorithm:lem:dependence}
	Consider systems \eqref{Setup and Preliminaries:system p}, $P \geq 2$ to be given. Suppose that applying Algorithm \ref{The covering algorithm:alg:abstract} for given maps $\Pi, \Theta: 2^\cP \to 2^\cP$ we have that $\cP_i \not = \emptyset$ with $i \geq 2$ holds for some $n \geq \overline{n}$ and all $\overline{n} \in \N_0$. Then for each system $p \in \cP_i$ there exists at least one system $q \in \cP_j$, $j < i$ such that $q \in \cIp(n)$. Moreover, in case cost functional \eqref{Setup and Preliminaries:cost functional infty} is used, we have
	\begin{align*}
		\ups = \argmin_{\up \in \bUpad(n, \xp(n), \emptyset)} \Jpinfty(\xp(n), \up) \not \in \bUpad(n, \xp(n), \ip(n)) \subsetneq \bUpad(n, \xp(n), \emptyset)
	\end{align*}
	and in case of cost functional \eqref{Setup and Preliminaries:cost functional p} we have
	\begin{align*}
		\ups = \argmin_{\up \in \bUpNpad(n, \xp(n), \emptyset)} \JpNp(\xp(n), \up) \not \in \bUpad(n, \xp(n), \ip(n)) \subsetneq \bUpNpad(n, \xp(n), \emptyset).
	\end{align*}
\end{lemma}
\begin{proof}
	Suppose that $\cP_i \not = \emptyset$ with $i \geq 2$ holds for some $n \geq \overline{n}$ and all $\overline{n} \in \N_0$ and fix $p \in \cP_i$ arbitrarily. Suppose furthermore that there exists no $q \in \cP_j$, $j < i$ such that $q \in \cIp(n)$ holds. Then, by the deordering rule $\Theta$ and Step 2a(i) we obtain that there exists $\overline{n} \in \N_0$ such that $\cP_i = \emptyset$ for all $n \geq \overline{n}$ contradicting our assumption. Hence, since $p \in \cP_i$ was chosen arbitrarily, we obtain that for each $p \in \cP_i$ there exists a system $q \in \cP_j$, $j < i$ such that $q \in \cIp(n)$ holds. \\
	Now, due to Step 2c(ii) and the fact that there exists a system $q \in \cIp(n)$ imposing constraints on system $p$ which are violated if $q \not \in \cIp(n)$ the assertion for both cost functionals \eqref{Setup and Preliminaries:cost functional infty} and \eqref{Setup and Preliminaries:cost functional p} follows.
\end{proof}

Now we can use Lemma \ref{The covering algorithm:lem:dependence} to answer the question under which conditions asymptotic stability can be shown. In particular, we first prove a necessary condition for asymptotic stability of \eqref{Setup and Preliminaries:system}.

\begin{theorem}\label{The covering algorithm:thm:no beta}
	Consider systems \eqref{Setup and Preliminaries:system p}, $P \geq 2$ to be given. Suppose that applying Algorithm \ref{The covering algorithm:alg:abstract} for all maps $\Pi, \Theta: 2^\cP \to 2^\cP$ we have that $\cP_i \not = \emptyset$, $i \geq 2$ holds for some $n \geq \overline{n}$ and all $\overline{n} \in \N_0$ with $\cIp(n, 1) \not = \emptyset$ for some $p \in \cP_i$. Then there exists no function $\beta \in \cKL$ such that \eqref{Stability:proposition:online alpha aposteriori:eq4} holds for all $n \in \N_0$.
\end{theorem}
\begin{proof}
	Fix maps $\Pi, \Theta: 2^\cP \to 2^\cP$. Then Lemma \ref{The covering algorithm:lem:dependence} states that for each system $p \in \cP_i$ there exists a system $q \in \cP_j$, $j < i$ such that $q \in \cIp(n)$. If for any $p \in \cP$ and any $n \in \N_0$ we have that $\bUpad(n, \xp(n), \ip(n)) = \emptyset$ or $\bUpNpad(n, \xp(n), \ip(n)) = \emptyset$ in case if cost functional \eqref{Setup and Preliminaries:cost functional infty} or \eqref{Setup and Preliminaries:cost functional p} are used, respectively, we are done since no admissible solution exists. Otherwise, we obtain ${\ups}^1(\cdot) \not = {\ups}^2(\cdot)$ with
	\begin{align*}
		{\ups}^1(\cdot) & = \argmin_{\up \in \bUpad(n, \xp(n), \ip(n))} \Jpinfty(\xp(n), \up), \quad {\ups}^2(\cdot) & = \argmin_{\up \in \bUpad(n, \xp(n), \emptyset)} \Jpinfty(\xp(n), \up)
	\end{align*}
	in case of cost functional \eqref{Setup and Preliminaries:cost functional infty} and with
	\begin{align*}
		{\ups}^1(\cdot) & = \argmin_{\up \in \bUpNpad(n, \xp(n), \ip(n))} \JpNp(\xp(n), \up), \quad {\ups}^2(\cdot) & = \argmin_{\up \in \bUpNpad(n, \xp(n), \emptyset)} \JpNp(\xp(n), \up)
	\end{align*}
	in case of cost functional \eqref{Setup and Preliminaries:cost functional p}.\\
	Hence, due to the fact that $\xp^{{\ups}^2}(k, x(n))$ for some $k$ violates a constraint imposed by system $q$ which is not violated by $\xp^{{\ups}^1}(k, x(n))$, we obtain that the open loop trajectories $\xp^{{\ups}^1}(\cdot, x(n))$ and $\xp^{{\ups}^2}(\cdot, x(n))$ differ. Using $\cIp(n, 1) \not = \emptyset$, we can conclude that there exists a $\delta_1 > 0$ such that $d_\X( \xp^{{\ups}^1}(1, \xp(n)), \xp^{{\ups}^2}(1, \xp(n)) ) > \delta_1$ holds. Since we always implement the first element of each optimal admissible control, we have that $d_\X ( \fp(\xp(n), {\ups}^1(0)), \fp(\xp(n), {\ups}^2(0)) ) > \delta_1$ holds. Now we have to consider two cases: If $\xp(n + 1) = \xpref$, then we can use the fact that the deviation $d_\X( \xp^{{\ups}^1}(1, \xp(\tilde{n})), \xp^{{\ups}^2}(1, \xp(\tilde{n})) ) > \delta_1$ will occur again for some $\tilde{n} > n$ due to the assumptions of the theorem. If $\xp(n + 1) \not = \xpref$, we immediately obtain the existence of a $\delta_2 > 0$ such that $\| \xp(n + 1) \|_{\xpref} > \delta_2$ holds. In either case, we obtain that there exists a time index $\tilde{n} > n$ such that $\| \xp(\tilde{n}) \|_{\xpref} > \delta = \min(\delta_1/2, \delta_2)$ holds. \\
	Now suppose there exists a function $\beta \in \cKL$ such that \eqref{Stability:proposition:online alpha aposteriori:eq4} holds for all $n \in \N_0$. Due to the $\cL$-property $\beta$ in its second argument, we have that for each $\varepsilon > 0$ there exists a $\hat{n} \in \N_0$ such that $\| x(n) \|_{\xref} < \varepsilon$ for all $n \geq \hat{n}$. Now we choose $\varepsilon < \delta$ and $\hat{n} \in \N_0$ accordingly. Since Lemma \ref{The covering algorithm:lem:dependence} holds for all $\overline{n} \in \N_0$, we can conclude that for $\tilde{n} > n \geq \overline{n} = \hat{n}$ the inequality $\| x(\tilde{n}) \|_{\xref} \geq \| \xp(\tilde{n}) \|_{\xpref} > \delta  > \varepsilon$ holds. This contradicts the existence of a function $\beta \in \cKL$ such that \eqref{Stability:proposition:online alpha aposteriori:eq4} holds for all $n \in \N_0$.\\
	Last, since the maps $\Pi$ and $\Theta$ were chosen arbitrarily, the argumentation holds for all choices of $\Pi$ and $\Theta$ which completes the proof.
\end{proof}

\begin{remark}
	Condition $\cIp(n, 1) \not = \emptyset$ in Theorem \ref{The covering algorithm:thm:no beta} is required since from $q \in \cIp(n)$ we can only conclude that $\xp^{{\ups}^1}(k_n, \xp(n))$ and $\xp^{{\ups}^2}(k_n, \xp(n))$ differ for some $k_n \geq 0$. According to the NMPC algorithm, only the first control element is implemented and we may face the situation that again $\xp^{{\ups}^1}(k_{n+1}, \xp(n + 1))$ and $\xp^{{\ups}^2}(k_{n+1}, \xp(n + 1))$ differ for some $k_{n+1} \geq k_n$. Now if $k_n > 0$ holds for all $n \in \N_0$, then system $p$ may be asymptotically stable.
\end{remark}

Turning from necessary to sufficient conditions we like to stress that the converse of Theorem \ref{The covering algorithm:thm:no beta} does not hold, not even in the special case that the conditions of Lemma \ref{The covering algorithm:lem:independence} hold for all $n \geq \overline{n}$ with $\overline{n} \in \N_0$. This conclusion is due to the fact that even if $\cP_2 = \emptyset$ we can only guarantee that a control which minimizes \eqref{Setup and Preliminaries:cost functional p} for all systems $p \in \cP$  can be computed without having to consider any other system $q \in \cP \setminus \{ p \}$, but not whether all systems are actually stable.

\begin{theorem}\label{The covering algorithm:thm:stability}
	Suppose that for given maps $\Pi, \Theta: 2^\cP \to 2^\cP$ we have that for a given initial value $x_0 \in \bX$ the set of admissible controls $\bUpNpad(n, \xp(n), \ip(n))$ is not empty for all $p \in \cP$ and all $n \in \N_0$. Suppose furthermore that there exist $\alpha_1^p, \alpha_2^p, \alpha_3^p \in \cK_\infty$, $\gamma \in \cK_\infty$ and $\alpha > 0$ such that inequalities \eqref{Stability:proposition:online alpha aposteriori distributed:eq2} and \eqref{Stability:proposition:online alpha aposteriori distributed:proof:eq1} hold for all $n \in \N_0$. Then there exists a function $\beta \in \cKL$ which only depends on $\alpha$, $\gamma$ and all $\alpha_1^p, \alpha_2^p, \alpha_3^p$, $p \in \cP$, such that \eqref{Stability:proposition:online alpha aposteriori:eq4} holds for all $n \in \N_0$. 
	
	Moreover, there exists an $\overline{n} \in \N_0$ such that for each $n \geq \overline{n}$ we either have that $\cP_i \not = \emptyset$, $i \geq 2$ holds with $\cIp(n, 1) = \emptyset$ for all $p \in \cP_i$ or $\cP_2 = \emptyset$.
\end{theorem}
\begin{proof}
	Using $x_0 \in \X$, $\bUpNpad(n, \xp(n), \ip(n)) \not = \empty$ for all $p \in \cP$ and all $n \in \N_0$ and Theorem \ref{The covering algorithm:thm:feasibility} we obtain that the closed loop solution $x(n) = ( x_1(n)^\top, \ldots, x_P(n)^\top )^\top$ exists for all $n \in \N_0$ and satisfies $x(n) \in \bX$ for all $n \in \N_0$. Now, since inequalities \eqref{Stability:proposition:online alpha aposteriori distributed:eq2} and \eqref{Stability:proposition:online alpha aposteriori distributed:proof:eq1} hold for all $n \in \N_0$, the existence of $\beta \in \cKL$ follows directly from Proposition \ref{Stability:proposition:online alpha aposteriori distributed2}. To show the existence of $\overline{n} \in \N_0$ such that for $n \geq \overline{n}$ we either have that $\cP_2 = \emptyset$ or $\cP_i \not = \emptyset$, $i \geq 2$ holds with $\cIp(n, 1) = \emptyset$ for all $p \in \cP_i$, suppose that $\cP_i \not = \emptyset$, $i \geq 2$ holds for some $n \geq \overline{n}$ and all $\overline{n} \in \N_0$ with $\cIp(n, 1) \not = \emptyset$ for some $p \in \cP_i$. Then, using Theorem \ref{The covering algorithm:thm:no beta} and the existence of $\beta \in \cKL$ we obtain a contradiction showing the assertion.
\end{proof}

\begin{remark}
	While the stability result of Theorem \ref{The covering algorithm:thm:stability} is given for the NMPC case without stabilizing terminal constraints or terminal costs, the only critical component in the proof of this theorem is the condition that $\bUpNpad(n, \xp(n), \ip(n)) \not = \emptyset$ which guarantees that the closed loop solution $x(\cdot) = ( x_1(\cdot)^\top, \ldots, x_P(\cdot)^\top )^\top$ exists and satisfies the state constraints. Hence, if instead of the existence conditions of $\alpha_1^p, \alpha_2^p, \alpha_3^p \in \cK_\infty$ and $\alpha > 0$ such that inequalities \eqref{Stability:proposition:online alpha aposteriori distributed:eq2} and \eqref{Stability:proposition:online alpha aposteriori distributed:proof:eq1} hold we impose other stability conditions -- e.g., the terminal constraint condition given in \citet{KG1988} or the terminal costs from \citet{ChAl1998}  -- then the same proof can be used to guarantee asymptotic stability of the closed loop.
\end{remark}


\section{Conclusion}
\label{Section:Conclusion}

We presented a generalized stability result for NMPC controllers without stabilizing terminal constraints or terminal costs. Moreover, we described an algorithm which allows us to generate a hierarchy of such controllers in a distributed non cooperative setting. Using only abstract priority and testing maps, we have shown necessary as well as sufficient conditions for stability of the closed loop.

Future research concerning the algorithm will certainly deal with the question how the priority and testing maps should be chosen to minimize the number of priority lists or to maximize the number of controllers that can be run in parallel. From the stability side an indeep analysis is required to apriori guarantee condition \eqref{Stability:proposition:online alpha aposteriori distributed:proof:eq1}. The availability of such a condition would then allow us to apriori guarantee Algorithm \ref{The covering algorithm:alg:abstract} to asymptotically stabilize the system. One idea in this direction is outlined in \citet[Section 7]{GW2010} and suggests the use of ISS small gain theorems to treat this problem.




\end{document}